\documentclass[11pt,letterpaper]{amsart}
\usepackage{tikz}
\usetikzlibrary{intersections, calc, arrows.meta}
\usepackage[utf8]{inputenc}
\usepackage{amsmath} 
\usepackage{amsmath}
\usepackage{amssymb} 
\usepackage{amsfonts} 
 \usepackage{amsthm}
 \usepackage{mathtools}
 \usepackage{latexsym}
 \usepackage{amscd}
 \usepackage[all]{xy}
 \usepackage{mathrsfs}
 \usepackage{yhmath}
 \usepackage{appendix}
\usepackage{amssymb}
\usepackage{mathtools}

\def\RR{\mathbb{R}}
\def\ZZ{\mathbb{Z}}
\def\CC{\mathbb{C}}

\def\ra{{\rightarrow}}

\def\del{\partial}

\def\odd{{\text{odd}}}

\DeclareMathOperator{\ind}{ind}

\DeclareMathOperator{\Fix}{Fix}

\DeclareMathOperator{\Per}{Per}

\theoremstyle{plain}
\newtheorem{thm}{Theorem}[section]
\newtheorem{prop}[thm]{Proposition}
\newtheorem{lemma}[thm]{Lemma}
\newtheorem{cor}[thm]{Corollary}

\theoremstyle{definition}
\newtheorem{dfn}[thm]{Definition}
\newtheorem{qtn}[thm]{Question}

\newtheorem{rmk}[thm]{Remark}

\usepackage{amssymb}
\usepackage{mathtools}

\def\RR{\mathbb{R}}
\def\ZZ{\mathbb{Z}}
\def\CC{\mathbb{C}}

\def\DD{\mathbb{D}}

\def\ra{{\rightarrow}}

\def\del{\partial}

\def\odd{{\text{odd}}}

 \usepackage{scalerel,stackengine}
\stackMath
\newcommand\reallywidecheck[1]{%
\savestack{\tmpbox}{\stretchto{%
  \scaleto{%
    \scalerel*[\widthof{\ensuremath{#1}}]{\kern-.6pt\bigwedge\kern-.6pt}%
    {\rule[-\textheight/2]{1ex}{\textheight}}
  }{\textheight}%
}{0.5ex}}%
\stackon[1pt]{#1}{\scalebox{-1}{\tmpbox}}%
}
\begin{document}

\pagestyle{plain}
\thispagestyle{plain}

\title[]{Area-preserving diffeomorphisms on the  disk and positive hyperbolic orbits}
\author[Masayuki ASAOKA]{Masayuki ASAOKA}\address[Masayuki Asaoka]{Faculty of Science and Engineering, Doshisha University,
 1-3 Tatara Miyakodani, Kyotanabe 610-0394, JAPAN.}
\email{masaoka@mail.doshisha.ac.jp}

\author[Taisuke SHIBATA]{Taisuke SHIBATA}
\address[Taisuke Shibata]{Research Institute for Mathematical Sciences, Kyoto University, Kyoto 606-8502,
JAPAN.}
\email{shibata@kurims.kyoto-u.ac.jp}

\date{\today}

\begin{abstract}
In this paper, we prove that if an area-preserving non-degenerate diffeomorphism on the open disk which extend smoothly to the boundary with non-degeneracy has at least 2 interior periodic  points, then there are  infinitely many positive hyperbolic  periodic points on the interior. As an application, we prove that if  a non-degenerate universally tight contact 3-dimentional lens space has a Birkhoff section of disk type and at least 3 simple periodic orbits, there are infinitely many simple  positive hyperbolic  orbits. In particular, we have that a non-degenerate dynamically convex contact 3-sphere has either  infinitely many simple  positive hyperbolic  orbits or exactly two simple elliptic orbits, which gives a refinement of the result proved by Hofer, Wysocki and Zehnder in \cite{HWZ2} under non-degeneracy.
\end{abstract}

\maketitle

\section{Introduction and the results}
\subsection{Introduction}
Area-preserving diffeomorphisms on the open annulus or the open disk  have been studied and  play important roles in 3-dimentional dynamics. They frequently arise as return maps on Birkhoff sections. For instance,  J. Franks \cite{Fr2,Fr3} showed that an area-preserving homeomorphism of the open 
annulus which has at least one periodic point  has infinitely many interior periodic points and as an application,  proved that every smooth Riemannian metric on $S^{2}$ with positive scalar curvature has infinitely many distinct closed geodesics. 
In the context of 3-dimensional Reeb flows, Hofer, Wysocki and Zehnder  \cite{HWZ2}  constructed a Birkhoff section of disk type  from a $J$-holomorphic curve  in a dynamically convex contact 3-sphere and proved by applying the Franks' result and Brouwer’s translation theorem that there are either  infinitely many simple periodic orbits or exactly two simple periodic orbits in a dynamically convex contact 3-sphere. 
The primary motivation of the paper is to refine the result of  Hofer, Wysocki and Zehnder and study periodic orbits in more detail. In particular, Our first theorem (Theomrem \ref{thm:main}) leads to the existence of infinitely many periodic orbits with the specific types called poisitive hyperboic if there are at least 3 periodic orbits.

\subsection{Diffeomorphisms on the disk}
Let $\Sigma$ be a surface. For a diffeomorphism $f$ of $\Sigma$, we call $p\in \Sigma$ a periodic point  with period $n(>0)$ if $p=f^{n}(p)$ and in addition $p\neq f^{m}(p)$ for any $0<m<n$. $f$ is called non-degenerate if for any $n$ and any  fixed point $p$ of $f^{n}$, the map $df^{n}:T_{p}\Sigma \to T_{p}\Sigma$ has no eigenvalue 1.

Consider a volume form $\omega$ on $\Sigma$. Let $f$ be a diffeomorphism  of $\Sigma$ with $f^{*}\omega=\omega$. A periodic point $p$ with period $n$ is called positive (resp. negative) hyperbolic if the eigenvalues of $df^{n};T_{p}\Sigma \to T_{p}\Sigma$ are positive (resp. negative) real numbers and elliptic if the eigenvalues of $df^{n};T_{p}\Sigma \to T_{p}\Sigma$ are of length 1. We note that since $f$ is area-preserving, any periodic point is either positive/negative hyperbolic or elliptic and if $f$ is non-degenerate, the conditions do not overlap each other.

Let $\DD$ be the closed unit disk and $\mathring{\DD}$ the interior. According to \cite{Fr2,Fr3,HWZ2}, it follows that an area and orientation preserving map on $\mathring{\DD}$ with finite area has either exactly two peiodic points or infinitely many periodic orbits. 

Our first  result is as follows.
\begin{thm}\label{thm:main}
Let  $\omega$ be a volume 2-form on $\mathring{\DD}$ with $\int_{\mathring{\DD}}\omega<+\infty$.
 Let $f$ be a non-degenerate diffeomorphism on $\DD$. If $f$ satisfies $f^{*}\omega=\omega$ and has  at least two periodic points on $\mathring{\DD}$, then the number of  positive hyperbolic periodic points on $\mathring{\DD}$ is infinite.
\end{thm}

As will be seen, the diffeomorphisms in Theorem \ref{thm:main} have highly compatibility with the return maps of Birkhoff sections of disk type near $J$-holomorphic curves in 3-dimensional Reeb flows.

\subsection{Applications to 3-dimensional Reeb flows}
 In this subsection, we observe how Theomre \ref{thm:main} is applied to 3-dimensional Reeb flows.

A closed contact three manifold $(Y,\lambda)$ is a pair of a closed contact three manifold $Y$ with a contact 1-form $\lambda$.   A  contact form $\lambda$ on $Y$ defines the Reeb vector field $X_{\lambda}$ and the contact structure $\xi=\mathrm{Ker}\lambda$.  A periodic orbit is a map $\gamma:\mathbb{R}/T_{\gamma}\mathbb{Z}\to Y$ satisfying $\Dot{\gamma}=X_{\lambda}\circ \gamma$ for some $T_{\gamma}>0$ and we write $\gamma^{p}$ for $p\in \mathbb{Z}$ as a periodic orbit of composing $\gamma$ with the natural projection $\mathbb{R}/pT_{\gamma}\mathbb{Z}\to \mathbb{R}/T_{\gamma}\mathbb{Z}$.  A periodic orbit $\gamma$ is simple if $\gamma$ is 
an embedding map and  non-degenerate if the return map $d\phi^{T_{\gamma}}|_{\xi}:\xi_{\gamma(0)}\to\xi_{\gamma(0)}$ has no eigenvalue $1$ where $\phi^{t}$ is the flow of $X_{\lambda}$. We call $(Y,\lambda)$ non-degenerate if all periodic orbits are non-degenerate.

\begin{dfn}
    Let $(Y,\lambda)$ be a contact three-manifold.  A Birkhoff section of disk type for $X_{\lambda}$ on a 3-manifold 
is a compact immersed disk $u:\DD \to Y$ such that
\item[(1).]$u(\DD\backslash \partial \DD)$ is embedded,
\item[(2).] $X_{\lambda}$ is transversal to $u(\DD\backslash \partial \DD)$,
\item[(3).] $u(\partial \DD)$ is tangent to a periodic orbit of $X_{\lambda}$,
\item[(4).] For every $x\in Y\backslash u(\partial \DD)$, there are $-\infty<t_{x}^{-}<0<t_{x}^{+}<+\infty$ such that $\phi^{t^{\pm}_{x}}(x)\in u(\DD)$ where $\phi^{t}$ is the flow of $X_{\lambda}$.
\end{dfn}

A Birkhoff sections of disk type in the context of 3-dimensional Reeb flows was first studied in \cite{HWZ2} and constructed  from $J$-holomorphic planes in dynamically convex contact 3-spheres. The notion of dynamically convex was introduced in \cite{HWZ2} as a generalization of strictly convex contact hypersurface in the 4-dimensional standard symplectic Euclidean space $(\RR^{4},\omega)$ (see Remark \ref{rmk:dyn}). In particular, strictly convex contact hypersurface in $(\RR^{4},\omega)$ is dynamically convex. 
A remarkable benefit of the existence of Berkhoff section is that the restriction of $d\lambda$ on $u(\mathring{\DD})$ define a volume form and the return map of $\phi:(u(\mathring{\DD}),d\lambda)\to (u(\mathring{\DD}),d\lambda)$ is orientation and volume preserving. In particular, as mentioned, it follows immediately from the Franks' theorem and Brouwer’s translation theorem  that there exist either two or infinitely simple periodic orbits.

\begin{rmk}\label{rmk:dyn}
A contact 3-manifold $(Y,\lambda)$ with $c_{1}(\xi)|_{\pi_{2}(Y)}=0$ is dynamically convex if any contractible periodic orbit has greater than or equal to 3 Conley-Zehnder index with respect to a trivialization induced by a binding disk (see \cite{HWZ2} for more details). If $(Y,\lambda)$ is  a dynamically convex contact 3-sphere, the contact structure must be tight. In addition since dynamical convexity is preserved under taking a finite cover,  the contact structure of a dynamically convex contact lens space must be universally tight. See \cite{HWZ1,HWZ2}.
\end{rmk}

\begin{rmk}
    Currently, a lot in 3-dimentional Reeb flows has been clarified by Embedded contact homology which was constructed by M. Hutchings. According to \cite{HT,CHrHL} if a contact 3-manifold $(Y,\lambda)$ has exactly two simple periodic orbit, then $(Y,\lambda)$ is dynamically convex and both of them are non-degenerate elliptic orbit. In addition, $Y$ is a lens space.  If $(Y,\lambda)$ is non-degenerate and not a lens space with exactly two simple orbits, there are infinitely many periodic orbits (see\cite{CHP,CoDR}).  
\end{rmk}

Consider a periodic orbit $\gamma$. If the eigenvalues of the return map $d\phi^{T_{\gamma}}|_{\xi}:\xi_{\gamma(0)}\to\xi_{\gamma(0)}$ are positive (resp. negative) real, $\gamma$ is called positive (resp. negative) hyperbolic. If the eigenvalues of the return map $d\phi^{T_{\gamma}}|_{\xi}:\xi_{\gamma(0)}\to\xi_{\gamma(0)}$ are on the unit circle in $\mathbb{C}$, $\gamma$ is called elliptic. As a refinment of Weinstein conjecture, it is natural to ask which kind of periodic orbit exists. 
For example,  D. Cristofaro-Gardiner, M. Hutchings and D. Pomerleano \cite{CHP} showed that  a non-degenerate $(Y,\lambda)$ with $b_{1}(Y)>0$ has as least one positive hyperbolic simple orbit by using Embedded contact homology and Monopole floer homology, and asked the following question.
\begin{qtn}\cite{CHP}
    Suppose that a non-degenerate  $(Y,\lambda)$ is not a lens space with exactly two simple elliptic orbit. Does $(Y,\lambda)$ has as least one positive hyperbolic simple orbit?
\end{qtn}

 As we will see later, our results support an affirmative answer to the question.

Before proceeding, we need to recall some notions.
Let $(Y,\lambda)$ be a contact 3-manifold. Consider a simple periodic orbit $\gamma:\RR/T_{\gamma} \ZZ\to Y$ and  $\gamma^{*}\xi \to Y$. Then the linearized flow  $d\phi^{t}|_{\xi}$ on the periodic orbit induces a flow on $\gamma^{*}\xi$ and hence on $(\gamma^{*}\xi\backslash 0)/\RR_{+}$. we write $(\gamma^{*}\xi\backslash 0)/\RR_{+}$ as $\mathbb{T}_{\gamma}$ and refer to the vector field induced by $d\phi^{t}$ on $\mathbb{T}_{\gamma}$ as linearized polar dynamics along $\gamma$.
As a set, the blown-up manifold is defined as $Y_{\gamma}:=(Y\backslash \gamma)\bigsqcup \mathbb{T_{\gamma}}$. $Y_{\gamma}$ has a  smooth structure of a manifold such that the Reeb vector field $X_{\lambda}$ extend smoothly to the  linearized polar dynamics on $\mathbb{T_{\gamma}}$ (see \cite[v1 Lemma A.1]{FHr}). It is easy to see that if $(Y,\lambda)$ is non-degenerate, any periodic orbit of $Y_{\gamma}$ is non-degenerate.

Let $u:\mathbb{D}\to Y$ be a Birkhoff section such that $u(\partial \mathbb{D})$ is tangent to $\gamma$. Then we can lift the map to $\tilde{u}:\mathbb{D}\to Y_{\gamma}=(Y\backslash \gamma)\bigsqcup \mathbb{T_{\gamma}}$ smoothly as follows. If $x\in \mathring{\mathbb{D}}$, then $\tilde{u}(x)=u(x)$. If $x\in \partial \mathbb{D}$, then $\tilde{u}(x):=\mathrm{pr}\circ du(\RR_{+} v)$ where $v$ is the outward unit vector at $x$ and $\mathrm{pr}$ is the projection $TY=\RR X_{\lambda} \oplus \xi \to \xi$.

\begin{dfn}\cite[c.f. Definition 1.6]{FHr}
A Birkhoff section $u:\mathbb{D}\to Y$ is $\partial$-strong if for the lift $\tilde{u}:\mathbb{D}\to Y_{\gamma}$, $\tilde{u}(\partial \mathbb{D})$ is transverse to the  linearized polar dynamics on $\mathbb{T_{\gamma}}$ and any trajectory on $\mathbb{T_{\gamma}}$ intersects $\tilde{u}(\partial \mathbb{D})$ infinitely many times in the future and in the past. Here $\gamma$ is the simple periodic orbit 
to which $u(\partial \mathbb{D})$ is tangent.
\end{dfn}

The following is an application of Theorem \ref{thm:main}.
\begin{thm}\label{main:cor}
 If a non-degenerate  contact 3-manifold $(Y,\lambda)$ 
admits a $\partial$-strong Birkhoff section of disk type 
 and has at least 3 simple periodic orbits, then there exists infinitely many simple positive hyperbolic orbits.  
\end{thm}
\begin{proof}[\bf Proof of Theorem \ref{main:cor}]
Since $(Y,\lambda)$ is non-degenerate and the Birkhoff section is $\partial$-strong, the return map on $u(\mathring{\mathbb{D}})$ is non-degenerate, area-preserving and orientation preserving  map with respect to $d\lambda$. In addition, it extends smoothly to the boundary of the disk with non-degeneracy and $\int_{u(\mathring{\mathbb{D}})}d\lambda<+\infty$ because of the Stokes' theorem. This implies that we can apply Theorem \ref{thm:main} to this map.
\end{proof}

Now, we recall the standard contact structure on a lens space $L(p,q)$.
Let  $p\geq q>0$  be mutually prime. The standard contact structure $\xi_{\mathrm{std}}$ on $L(p,q)$ is defined as follows.
Consider a contact 3-sphere $(\partial B(1),\lambda_{0}|_{\partial B(1)})$ where $\partial B(1)=\{(z_{1},z_{2})\in \mathbb{C}^{2}||z_{1}|^{2}+|z_{2}|^{2}=1\}$, $\lambda_{0}=\frac{i}{2}\sum_{i=1,2}(z_{i}d\Bar{z_{i}}-\Bar{z_{i}}dz_{i})$. The action $(z_{1},z_{2})\mapsto (e^{\frac{2\pi i}{p}}z_{1},e^{\frac{2\pi iq}{p}}z_{2})$ preserves $(\partial B(1),\lambda_{0}|_{\partial B(1)})$ and the tight contact structure. Hence we have the quotient space which is a contact manifold and write  $(L(p,q),\lambda_{p,q})$, $\xi_{\mathrm{std}}=\mathrm{Ker}\lambda_{p,q}$.
If a contact manifold is contactomorphic to a universally tight lens space, we can simplify the assumption  as follows.

\begin{thm}\label{thm:tight}
 Let $\lambda$ be  a non-degenerate  contact form on $(L(p,q),\xi_{\mathrm{std}})$.
If $(L(p,q),\lambda)$ admits a  Birkhoff section of disk type and  has at least 3 simple periodic orbits, then there are infinitely many simple positive hyperbolic  orbits.
\end{thm}

\begin{rmk}\label{rmk:tight}
 In \cite{HrS1,HrLS}, necessary and sufficient conditions for $(L(p,q),\lambda)$ with $\mathrm{Ker}\lambda=\xi_{\mathrm{std}}$ having a Birkhoff section of disk type are given.
\end{rmk}

 The next proposition allows us to apply  Theorem \ref{main:cor} to Theorem \ref{thm:tight} and hence Theorem \ref{thm:tight} follows immediately.

\begin{prop}\label{main;prop}
 Let $\lambda$ be  a non-degenerate  contact form on $(L(p,q),\xi_{\mathrm{std}})$.
If $(L(p,q),\lambda)$ admits a  Birkhoff section of disk type, then there is a $\partial$-strong Birkhoff section of disk type with the same binding.
\end{prop}
Proposition \ref{main;prop} is proved in the last section of this paper.

The existence of a Birkhoff section of disk type has been studied under dynamical convexity. To explain it, we introduce some notions which are also used in Section 3.

\begin{dfn}
    A knot $K\subset Y$ is called $p$-unknotted if there exists an immersion $u:\mathbb{D}\to Y$ such that $u|_{\mathrm{int}(\mathbb{D})}$ is embedded and $u|_{\partial \mathbb{D}}:\partial \mathbb{D}\to K$ is a $p$-covering map.
    \end{dfn}
 \begin{dfn}\cite[cf. Subsection 1.1]{BE}
    Assume that a knot $K\subset Y$ is $p$-unknotted, transversal to $\xi$ and oriented by the co-orientation of $\xi$. Let $u:\mathbb{D}\to Y$ be  an immersion  such that $u|_{\mathrm{int}(\mathbb{D})}$ is embedded and $u|_{\partial \mathbb{D}}:\partial \mathbb{D}\to K$ is a $p$-covering map. Take a non-vanishing section $Z:\mathbb{D}\to u^{*}\xi$ and consider the immersion $\gamma_{\epsilon}:t\in \mathbb{R}/\mathbb{Z} \to \mathrm{exp}_{u(e^{2\pi i t})}(\epsilon Z(u(e^{2\pi i t})))\in Y\backslash K$ for small $\epsilon>0$. 
    
    Define the rational self-linking number $\mathrm{sl}(K,u)\in \mathbb{Q}$ as
    \begin{equation*}
        \mathrm{sl}(K,u)=\frac{1}{p^{2}} (\mathrm{algebraic\,\,intersection\,\,number\,\,of}\,\, \gamma_{\epsilon}\,\,\mathrm{with}\,\,u)
    \end{equation*}
    If  $c_{1}(\xi)|_{\pi_{2}(Y)}=0$, $\mathrm{sl}(K,u)$ is independent of $u$. Hence  we write $\mathrm{sl}(K)$.
\end{dfn}
\begin{rmk}
In generall, (rational) self-linking number is defined for rationally null-homologous knot by using a (rational) Seifert surface. See \cite{BE}.
\end{rmk}

We assume that lens spaces $L(p,q)$ contain $S^{3}$ as a lens space with $p=1$.

\begin{thm}\label{fundament}\cite[Theorem 1.7, Corollary 1.8]{HrS2}
If $\lambda$ is any dynamically convex contact form on $L(p,q)$, then for every
$p$-unknotted simple orbit $\gamma$ with $\mathrm{sl}(\gamma)=-\frac{1}{p}$, $\gamma^{p}$ must bound  
a disk which is a Birkhoff section. Moreover, this
Birkhoff section is a page of a rational open book decomposition of $L(p,q)$  such that all pages are Birkhoff sections. 
\end{thm}

As an immediate corollary of Theorem \ref{thm:tight} and Theorem \ref{fundament}, we have

\begin{cor}
    Let  $(L(p,q),\lambda)$ be a non-degenerate dynamically convex. If there is a $p$-unknotted simple orbit $\gamma$ with $\mathrm{sl}(\gamma)=-\frac{1}{p}$, then there are   either infinitely many simple  positive hyperbolic  orbits or  exactly two simple elliptic orbits.
\end{cor}

Whether a dynamically convex lens space $(L(p,q),\lambda)$ has $p$-unknotted simple orbit $\gamma$ with $\mathrm{sl}(\gamma)=-\frac{1}{p}$ has been studied and is partially known. In particular it depends on the contact structure. First of all, it was proved by Hofer Wysocki Zehnder \cite{HWZ2} that any dynamically convex $(S^{3},\lambda)$ must have a  $1$-unknotted simple orbit $\gamma$ with $\mathrm{sl}(\gamma)=-1$ , and recentry Hryniewicz and Salomão \cite{HrS2} showed the same result for $L(2,1)$ by developing the original technique and after that  Schneider \cite{Sch} generalized it to $(L(p,1).\xi_{\mathrm{std}})$. On the other hand, the second  author \cite{Shi} showed in that non-degenerate dynamically convex $(L(p,p-1),\lambda)$ with $\lambda$ must have a  $p$-unknotted simple orbit $\gamma$ with $\mathrm{sl}(\gamma)=-\frac{1}{p}$ by using Embedded contact homology. 
In summary,
\begin{cor}
     Assume $(Y,\lambda)$ be a dynamically convex non-degenerate contact 3-manifold such that $Y$ is diffeomorphic to $L(p,p-1)$ for some $p$. Then then there are  either infinitely many simple  positive hyperbolic  orbits or  exactly two simple elliptic orbits.
\end{cor}
\begin{rmk}
    It follows from \cite{HrS2} that any dynamically convex $(L(p,q),\lambda)$ with even $p$ has an elliptic orbit. Combining with \cite{Shi1,Shi2}, we have that any non-degenerate dynamically convex $(L(p,q),\lambda)$ with at least 3 simple periodic orbits has a simple positive hyperbolic orbit.
\end{rmk}
We end this section with the following question.
\begin{qtn}
    Let $(Y,\lambda)$ be a non-degenerate contact 3-manifold. Assume $(Y,\lambda)$ is not a lens space with exactly two simple elliptic orbit. Does $(Y,\lambda)$ have infinitely many simple positive hyperbolic orbits?
\end{qtn}
\subsection*{Acknowledgement}
TS would like to thank his advisor Professor Kaoru Ono for his support. MA was supported by JSPS KAKENHI Grants 22K03302. TS was supported by JSPS KAKENHI Grants  JP21J20300.

\section{Proof of Theorem \ref{thm:main}}
In this section, we prove Theorem \ref{thm:main}. For the purpose, we start in general situations.

Let $\Sigma$ be a surface.
We denote the set of fixed point of $f$ by $\Fix(f)$, the set $\bigcup_{n:\odd}\Fix(f^n)$  of periodic points with odd period by  $\Per^\odd(f)$ and the set of positive hyperbolic periodic points by $\Per_{h+}(f)$.
For any isolated fixed point $p$ of $f$, let $\ind(p,f)$ be the fixed point index of $f$. Notice that the fixed point index of a fixed point  in the boundary is defined by the fixed point index for then extension $\tilde{f}$ of $f$ to an open manifold $\tilde{\Sigma}$ such that $\tilde{f}(\tilde{\Sigma})=\Sigma$.
When the diffeomorphism is non-degenerate, any periodic points at the boundary are `positive half-saddles', whose fixed point index is $0$ (contracting along the boundary circle) or $-1$ (expanding along the boundary circle).

Fix an area preserving diffeomorphism $f$ which is non-degenerate and is isotopic to the identity map.
The following lemmas reduce Theorem \ref{thm:main} to finding infinitely many periodic points with odd period of $f$ or $f^2$.
\begin{lemma}
\label{lemma:odd}
Let $f$ be a non-degenerate diffeomorphism of compact surface $\Sigma$
 which is area-preserving on the interior and isotopic to the identity map.
In addition, we assume that $f$ is area preserving on the interior with respect to a volume form defined on the interior.
If $\Per^\odd(f)$ is infinite,
 then $\Per^\odd(f) \cap \Per_{h+}(f)$ is infinite.
\end{lemma}
\begin{proof}[\bf Proof of Lemma \ref{lemma:odd}]
Suppose that $f$ admits infinite number of periodic points with odd period
 but only finite number of them are positive hyperbolic.
Let $K$ be the number of positive hyperbolic periodic points of $f$
 and periodic points in the boundary of $\Sigma$ with odd period.
Put $\Lambda_n=\Fix(f^n) \cap (\Per_{h+}(f) \cup \del \Sigma)$.
Then, we have
\begin{equation*}
 \sum_{p \in \Lambda_n}\ind(p,f^n) \geq -K
\end{equation*}
for any odd $n$.
Put $L=\sum_{i\geq 0} (-1)^i \dim H_i(S)$.
Since $f$ is isotopic to the identity map,
 the Lefschetz number of $f^n$ equals to $L$ for any $n \geq 1$.
There are infinitely many periodic points with odd period
 which are not positive hyperbolic
 and the boundary of $\Sigma$ contains only finitely many periodic points.
Hence, we can take periodic points $p_1,\dots,p_{K+L+1}$
 in $\Per^\odd(f) \setminus (\Per_{h+}(f) \cup \del \Sigma)$.
Let $N$ be the product of the periods of $p_1,\dots,p_{K+L+1}$
Then $N$ is odd and
 any point $p \in \Fix(f^N) \setminus \Lambda_n$
 satisfies $\ind(p,f^N)=1$
We have
\begin{align*}
 \sum_{p \in \Fix(f^N)} \ind(p,f^N)
 &= \sum_{p \in \Fix(f^N) \setminus \Lambda_n}\ind(p,f^N)
  +\sum_{p \in \Lambda_n} \ind(p,f^N)\\
 & \geq \sum_{i=1}^{K+L+1}\ind(p_j,f^N) +\sum_{p \in \Lambda_n} \ind(p,f^N)\\
 & \geq K+L+1-K \geq L+1.
\end{align*}
This contradicts to the Lefschetz fixed point theorem
 since the Lefschetz number of $f^N$ equals to $L$.
\end{proof}

\begin{lemma}
\label{lemma:odd 2}
Let $f$ be a non-degenerate diffeomorphism of compact surface $\Sigma$
 which is area-preserving on the interior and isotopic to the identity map.
Suppose that $\Per^\odd(f^2)$ is infinite
 then $\Per_{h+}(f)$ is infinite.
\end{lemma}
\begin{proof}[\bf Proof of Lemma \ref{lemma:odd 2}]
If $\Per^{\odd}(f)$ is infinite,
 then $\Per_{h+}(f)$ is infinite by Lemma \ref{lemma:odd} again.
Suppose that $\Per^{\odd}(f)$ is finite.
By Lemma \ref{lemma:odd}, the set $\Per_{h+}(f^2) \cap \Per^\odd(f^2)$
 is infinite.
This implies that
 $\Per_{h+}(f^2) \cap (\Per^\odd(f^2) \setminus \Per^\odd(f))$ is infinite.
The period of any point in 
 $\Per_{h+}(f^2) \cap (\Per^\odd(f^2) \setminus \Per^\odd(f))$
 is twice of an odd number,
 and hence, such a point is positively hyperbolic.
Hence, $\Per_{h+}(f)$ is infinite.
\end{proof}

Let $A$ be the annulus $A=S^1 \times [0,1]$
 and $\pi:\RR \times [0,1] \ra A$ the universal covering.
For a homeomorphism $\tilde{f}$ of $\RR \times [0,1]$
 and $\tilde{x} \in  \times [0,1]$,
 we define t{\it the translation number} $\tau(\tilde{x})$ by
\begin{equation*}
 \tau(\tilde{x})
=\lim_{n \ra \infty}\frac{\tilde{f}^n(\tilde{x})_1-\tilde{x}_1}{n}
\end{equation*}
 if the limit exists, where
 $\tilde{f}^n(\tilde{x})_1$ and $\tilde{x}_1$ are the first coordinates
 of $\tilde{f}^n(\tilde{x})$ and $\tilde{x}$.
For a homeomorphism $f$ of $A$ and $x \in A$,
 take lifts $\tilde{f}$ of $f$ and $\tilde{x}$ of $x$ to $\RR \times [0,1]$.
Then, the translation number $\tau(\tilde{x})$ modulo $\ZZ$
 does not depend on the choice of lift if it exists.
We define {\it the rotation number} $\rho(x)$ by
 $\rho(x)=\tau(\tilde{x})+\ZZ$.
To finding infinitely many periodic points with odd period,
 we use the following fixed point theorem by Franks.
\begin{thm}\cite[Cororraly 2.4]{Fr1}\cite[Theorem 2.1]{Fr2}
\label{thm:Franks}
Let $f$ be a homeomorphism of $A$ which is isotopic to 
 the identity map such that any point of $A$ is chain recurrent.
Suppose that a lift of $f$ to $\RR \times [0,1]$ admits
 points $\tilde{x}, \tilde{y} \in \RR \times [0,1]$
 such that the translation numbers $\tau(\tilde{x}), \tau(\tilde{y})$
 exists and $\tau(\tilde{x})<\tau(\tilde{y})$.
Then for any pair $(m,n) $of co-prime integers with $n \geq 1$
 and $\tau(\tilde{x})<m/n<\tilde{y}$,
 there exists $\tilde{x}_{m/n} \in \RR \times [0,1]$ 
 such that $\tilde{f}^n(\tilde{x}_{m/n})=T^m(\tilde{x}_{m/n})$,
 where $T:\RR \times [0,1] \ra \RR \times [0,1]$ is
 the translation given by$T(x,y)=(x+1,y)$,
In particular, $\pi(\tilde{x}_{m/n})$
 is a periodic point of $f$ whose period is $n$.
\end{thm}
\begin{cor}
\label{cor:Franks}
Let $f$ be a homeomorphism of $A$ which is isotopic to 
 the identity map such that any point of $A$ is chain recurrent.
If there exists $x,y \in A$ such that $\rho(x) \neq \rho(y)$
 then, $f$ has infinitely many periodic points of odd period.
\end{cor}
\begin{rmk}
    See for the definition of chain recurrence \cite{Fr1}. Note that for a diffeomorphism $f$ on $\DD$ in Theorem \ref{thm:main}, any point in $\DD$ is chain recurrent. This follows immediately from the Poincare recurrence theorem.
\end{rmk}

Now, we prove Theorem \ref{thm:main}.
Let $f$ be an area preserving diffeomorphism of $\DD$ on the interior
 which is non-degenerate and is orientation preserving.
We show that $f$ or $f^2$ has infinitely many periodic points with odd period.
Then, Lemmas \ref{lemma:odd} and \ref{lemma:odd 2} imply that
 $f$ admits infinitely many positive hyperbolic periodic points.

Recall that
 the fixed point index of any possible fixed point on the bounday of $\DD$ 
 is $0$ or $-1$.
By the Lefschetz fixed point theorem,
 $f$ admits a fixed point $p_*$ in the interior of $\DD$
 with $\ind(p_*,f)=1$.
Take the blow-up annulus $A_{p_*}$at $p_*$
 and lift the diffeomorphism $f$ to
 a diffeomorphism $\hat{f}$ on $A_{p_*}$.
Let $\rho_D$ be the rotation number of $\hat{f}$ along the boundary
 component of $A_{p_*}$ which corresponds to the boundary of $\DD$
 and $\rho_{p_*}$ the rotation number of $\hat{f}$ along
 the boundary component of $A_{p_*}$
 which corresponds to $p_*$.
Since the fixed point index of $p_*$ is one,
 $p_*$ is either negative hyperbolic or elliptic.
We have $\rho_{p_*}=1/2$ in the former case
 and $\rho_{p_*}$ is irrational in the latter case.

The easiest case is that $\rho_D \neq \rho_{p_*}$.
In this case,
 Corollary \ref{cor:Franks} implies that
 $\hat{f}$, and hence, $f$ has infinitely many periodic points of odd period.

The second case is that $\rho_D= \rho_{p_*}$ and they are irrational.
By the assumption, $f$ has at least two periodic points.
Hence, there exists a periodic point $q_*$ of $f$
 different from $p_*$.
In the blow-up annulus $A_{p_*}$, the periodic point $q_*$ has
 rational rotation number for $\hat{f}$.
Since $\rho_D=\rho_{p_*}$ is irrational, we can apply
 Corollary \ref{cor:Franks}
 and obtain infinitely many periodic points of odd period.

The last case is that $\rho_D= \rho_{p_*}=1/2$.
In this case, $p_*$ is negative hyperbolic.
If $f$ has a fixed point $q_*$ different from $p_*$,
 then  the lift to the blow up annulus at $p_*$ has a fixed point $q_*$,
 whose rotation number is zero by definition,
 and the boundary components whose rotation number is $1/2$.
By Corollary \ref{cor:Franks},
 $\hat{f}$, and hence, $f$ has infinitely many periodic points of odd period.
Suppose that $f$ has no fixed point other than $p_*$.
Since $p_*$ is a positive hyperbolic fixed point of $f^2$,
 we have $\ind(p_*,f^2)=-1$.
Recall that the fixed point index of any fixed point
 in the boundary is non-positive.
By the Lefschetz fixed point theorem,
 $f$ must have a $2$-periodic point $r_*$ with $\ind(r_*,f^2)=1$.
The rotation number of $f^2$ along the boundary is $0$
 and the rotation number of the blow up of $f^2$ at $r_*$ is
 $1/2$ or irrational.
Therefore, $f^2$ has infinitely many periodic points with odd period
 by Corollary \ref{cor:Franks}.
Now, Lemma \ref{lemma:odd 2} completes the proof.

\section{Proof of Proposition \ref{main;prop}}
 In the assumptions, we may replace  the given Birkhoff section of disk type to one coming from a $J$-holomrphic plane as follows.  Let $\gamma$ be the simple orbit to which the given Birkhoff section is tangent. Note that by the assumption, any periodic orbit in $L(p,q))\backslash \gamma$ is not contractible in $L(p,q))\backslash \gamma$. 
According to \cite[Theorem 1.12, $\mathrm{i})\to \mathrm{iii})$]{HrLS}, $\gamma$ is $p$-unknotted and $\mathrm{sl}(\gamma)=-\frac{1}{p}$. In addition, the Conley-Zehnder index  of $\gamma^{p}$ with respect to a trivialization induced by a binding disk is at least 3.  
Now, we recall the proof of \cite[Theorem 1.12, $\mathrm{iii})\to \mathrm{i})$]{HrLS}. To explain it, we consider an almost complex structure $J$ on $\RR \times L(p,q)$ which 
satisfies $J\xi_{\mathrm{std}}=\xi_{\mathrm{std}}$, $J(\partial_{t})=X_{\lambda}$,  $d\lambda$-compatible and $\RR$-invariant,  where $t$ is the coordinate of $\RR$. Let  $\mathrm{pr}:\RR \times L(p,q) \to L(p,q)$ be the projection
In the proof, they find an almost complex structure $J$ as above and a $J$-holomorphic plane $h:(\CC,j)\to (\RR \times Y, J)$ such that $\mathrm{pr}\circ h(re^{2\pi t})\to \gamma(pT_{\gamma}t)$ as $r \to +\infty$  and in addition $\mathrm{pr}( \overline{h(\CC)})$ becomes a Birkhoff section. More precisely, there is  a $C^{1}$ Birkhoff section of disk type $u:\DD \to L(p,q)$ such that $\mathrm{pr}( \overline{h(\CC)})=u(\DD)$ as sets (see \cite[v1, Lemma C.3]{FHr}).

Having a $C^{1}$ Birkhoff section of disk type $u:\DD \to L(p,q)$ coming from a $J$-holomorphic plane, it follows from \cite[v1, Lemma C.6.]{FHr} that we can find  a  $C^{\infty}$ $\partial$-strong Birkhoff section $u':\DD \to L(p,q)$ which is arbitrary close to $u$ in $C^{1}$-topology. Which completes the proof of Proposition \ref{main;prop}. We note that although originally \cite[Lemma C.3]{FHr} and \cite[Lemma C.6.]{FHr} are discussed on $S^{3}$, we can apply the proofs to $L(p,q)$ in exactly the same way.

\end{document}